\documentclass[11pt,bezier,amstex]{article}  % include bezier curves
\usepackage[textsize=small]{todonotes}       % includes TO DO LIST

\usepackage{color}              % Need the color package
\usepackage{graphicx,setspace}
\usepackage[]{amsmath}
\usepackage{amssymb}
\usepackage{fullpage}
\usepackage{hyperref}
\usepackage{enumerate, amsfonts, amsthm, bbm}

\definecolor{MyDarkBlue}{rgb}{0,0.08,0.50}
\definecolor{BrickRed}{rgb}{0.65,0.08,0}

\hypersetup{
colorlinks=true,       % false: boxed links; true: colored links
    linkcolor=MyDarkBlue,          % color of internal links
    citecolor=BrickRed,        % color of links to bibliography
    filecolor=red,      % color of file links
    urlcolor=cyan           % color of external links
}
\numberwithin{equation}{section}

\usepackage{amssymb,graphicx,tabularx}

\def\frac#1#2{{{\lower.6ex
\hbox{$\scriptstyle#1$}}\over
{\raise.7ex
\hbox{$\scriptstyle#2$}}}}
\def\eps{{\varepsilon}}

\def\Frac#1#2{\frac
{
 {\raise.6ex
 \hbox{$\disp#1$}}
}
{
 {\lower.6ex
 \hbox{$\disp#2$}}
 }
}

\newtheorem{theorem}{Theorem}[]

\newtheorem{lemma}[theorem]{Lemma}

\newtheorem{corollary}[theorem]{Corollary}

\newtheorem{proposition}[theorem]{Proposition}
\theoremstyle{definition}

\def\Remark{{\noindent\sl Remark}}
\def\frac#1#2{{#1\over#2}}
\def\half{{1\over2}}
\def\thalf{{\textstyle\half}}

\def\tsqrt#1{{\textstyle\sqrt{#1}}}

\def\dep{{\rm dep}}

\def\var{\mathop{\rm var}\nolimits}

\def\artan{\mathop{\rm{artan}}\nolimits}

\def\bbbR{{\Bbb{R}}}

\def\Rpos{\bbbR_+}

\def\d{{\rm d}}
\def\e{{\rm e}}
\def\E{{\rm E}}
\def\Pr{{\rm Pr}}
\def\tC{\widetilde{C}}

\def\eps{\epsilon}

\def\Ueps{U_\eps}

\newcommand{\eqan}[1]{\begin{align} #1 \end{align}}

\newcommand{\RRR}{\mathcal{D}}
\newcommand{\h}{h}
\newcommand{\f}{f(\eta,\beta)}
\begin{document}
\newcommand{\rur}{\rule{\breite}{0.06mm}\vspace{1mm}\\}

\title{BRAVO for many-server QED systems\\
with finite buffers}
\author{
Daryl J. Daley\footnote{Department of Mathematics and Statistics, The University of Melbourne, dndaley@gmail.com.},
Johan S.H.~van Leeuwaarden\footnote{Department of Mathematics and CS, Eindhoven University of Technology. j.s.h.v.leeuwaarden@tue.nl.},
 Yoni Nazarathy\footnote{School of Mathematics and Physics, The University of Queensland. y.nazarathy@uq.edu.au.}
}

\maketitle

\abstract{
%%%%%%%%%%%%%%
This paper demonstrates the occurrence of the feature called BRAVO (Balancing Reduces
Asymptotic Variance of Output) for the departure process of a
finite-buffer
%stationary
Markovian many-server system in the QED (Quality and Efficiency-Driven) heavy-traffic regime. The results are based on evaluating the limit of a formula for the asymptotic variance of death counts in finite birth--death processes.
}

%%%%%%%%%%%%%%%%%%%%%%%%%%%%%%%%%%%%%%%%%%%%%%%%%%
%%%%%%%%%%%%%%%%%%%%%%%%%%%%%%%%%%%%%%%%%%%%%%%%%%
%%%%%%%%%%%%%%%%%%%%%%%%%%%%%%%%%%%%%%%%%%%%%%%%%%
%%%%%%%%%%%%%%%%%%%%%%%%%%%%%%%%%%%%%%%%%%%%%%%%%%
%%%%%%%%%%%%%%%%%%%%%%%%%%%%%%%%%%%%%%%%%%%%%%%%%%
%%%%%%%%%%%%%%%%%%%%%%%%%%%%%%%%%%%%%%%%%%%%%%%%%%
%%%%%%%%%%%%%%%%%%%%%%%%%%%%%%%%%%%%%%%%%%%%%%%%%%
%%%%%%%%%%%%%%%%%%%%%%%%%%%%%%%%%%%%%%%%%%%%%%%%%%
%%%%%%%%%%%%%%%%%%%%%%%%%%%%%%%%%%%%%%%%%%%%%%%%%%

%\tableofcontents
\section{Introduction}
\label{sec:intro}
The QED (Quality and Efficiency Driven) regime for many-server systems combines large capacity with high utilization while maintaining satisfactory system performance.
In the QED regime the arrival rate $\lambda$ and the number of servers $s$ are scaled in such a way that while they both increase towards infinity, the traffic intensity $\rho=\rho_s=\lambda/\mu$ (assuming service rate $\mu/s$ per server) approaches one and
\begin{equation}\label{scaling}
(1-\rho_s)\sqrt{s}\rightarrow \beta, \quad \beta\in(-\infty,\infty).
\end{equation}
Halfin and Whitt \cite{halfin1981htl} introduced the QED regime for the $GI/M/s$ system. Under the scaling \eqref{scaling}, assuming $\beta>0$, the stationary probability of delay was shown to converge to a non-degenerate limit, bounded away from both zero and one. Limit theorems for other, more general systems were obtained in \cite{garnett2002designing, janssen2011refining, jelenkovic2004heavy, maglaras2004diffusion,mandelbaum2008queues,reed2009g}, and for all those cases, the limiting probability of delay remains in the interval $(0,1)$. In fact, not only the probability of delay, but several other performance characteristics or objective functions are shown to behave (near) optimally in the QED regime (see \cite{borst2004dimensioning, gans2003telephone}).

Associated with the near optimal behaviour is the fact that the process $Q_s(\cdot)$ that counts the total number of customers in the system at any time exhibits relatively small fluctuations. Halfin and Whitt \cite{halfin1981htl} showed for the $GI/M/s$ system that under \eqref{scaling} a sequence of standardized processes $X_s(\cdot)$, with  $X_s(t):=(Q_s(t)-s)/\sqrt{s}$, converges as $s\to\infty$ to a diffusion process $X(\cdot)$. This diffusion process behaves like a Brownian motion with drift above zero and like an Ornstein--Uhlenbeck process below zero. The interpretation of this scaled process is as follows: In case, $X_s(t)>0$, it represents the scaled number of customers waiting for service, whereas in case $X_s(t)<0$, it represents the scaled number of idle servers. This result shows that the natural scale that emerges is of the order $\sqrt{s}$: Specifically, both the queue length and the number of idle servers in the system are of the order~$\sqrt{s}$.

Due to its favorable behavior, the many-server QED regime has been a major focal point of applied probability and stochastic operations research in the past 30 years. The many extensions of the Halfin--Whitt exposition in  \cite{halfin1981htl} have lead to theoretical advances in the areas of stochastic-process limits and asymptotic dimensioning. From an operational point of view, the QED regime has found many applications in the planning, analysis and optimization of queueing, inventory and service systems (see for example \cite{armony2004customer}).

In this paper we explore the presence of a BRAVO effect in the QED regime. BRAVO is short for Balancing Reduces Asymptotic Variance of Outputs. This again would be a favourable property of QED, this time for the departure process
$N_{\rm dep}(0,\cdot]$, where $N_{\rm dep}(0,t]$ counts the number of serviced customers during the time interval $(0,t]$. The study of departure or output processes of queues has a long tradition, see for example the classic surveys \cite{daley1976queueing} and \cite{disney1985queueing}, yet only recently the surprising phenomenon of the BRAVO effect has been reported. The BRAVO effect is captured in terms of the asymptotic ratio of the variance and the mean of the departure process
\eqan{\label{limrat}
\RRR := \lim_{t\to\infty} \frac{\var \big( N_{\rm dep}(0,t] \big)}{\E \big(N_{\rm dep}(0,t]\big)}.
}
For Poisson processes $\RRR=1$, and more generally for renewal processes, $\RRR$ equals the ratio of the variance of the renewal lifetime and the square of the mean lifetime. Thus it is initially surprising that for $M/M/1/K$ systems, $\RRR$ is minimized when the arrival rate $\lambda$ is equal to the service rate $\mu$ with a minimum equal to $\frac{2}{3} + o_K(1)$, where $o_K(1)$ is a term that vanishes as $K \to \infty$. This was shown in \cite{NazarathyWeiss0336}. Further, when $K=\infty$, it is well known that $\RRR=1$ whenever $\lambda \neq \mu$, yet it was shown in \cite{al2011asymptotic} that in the critical case that $\lambda=\mu$, $\RRR = 2(1-2/\pi) \neq \frac 23$. The work in \cite{al2011asymptotic} goes further, generalizing this $M/M/1$ result to $GI/GI/1$ systems and even multi-server $GI/GI/s$ systems with a finite bounded number of servers $s$. Hence, by BRAVO we mean that $\RRR< 1$ when $\rho=1$. For overviews of BRAVO results we refer to \cite{daley2011revisiting} and \cite{nazarathy2011variance}.

In this paper we study the BRAVO effect in the QED regime by equipping the
many-server $M/M/s$ system with a finite waiting capacity $K$. In order to create a finite-capacity effect in the QED regime that is neither dominant nor negligible, it is plausible to assume that $K \approx \eta\sqrt{s}$, because the natural scale of the queue length is $\sqrt{s}$. More precisely, we study a sequence of systems in which both $K$ and $s$ grow in such a way that
\begin{equation}
\label{scaling2}
\frac{K}{\sqrt{s}}\rightarrow \eta
\end{equation}
for some positive $\eta$.
A similar threshold $K\approx\eta\sqrt{s}$ in the context of many-server systems in the QED regime has been considered in \cite{armony2004customer, janssenScaled,massey2005asymptotically,whitt2004diffusion,whitt2005heavy}.
Hence, in addition to the parameter $\beta$ in \eqref{scaling} describing the scaled shortfall from one in the system capacity, our system includes a parameter $\eta$ describing the relative buffer size. Our result on the BRAVO effect in the QED regime is in terms of
\eqan{\label{double}
\RRR_{\beta, \eta} :=\lim_{s,K\to\infty} \lim_{t\to\infty} \frac{\var \big( N_{\rm dep}(0,t] \big)}{\E \big(N_{\rm dep}(0,t]\big)},
}
where the outer limit is taken under the constraints at \eqref{scaling} and \eqref{scaling2}.

Our analysis that leads to explicit representations for
 $\RRR_{\beta,\eta}$ is based on the following general result for output processes of birth--death processes.  Consider a finite, irreducible, birth--death process $Q(\cdot)$ on $\{0,1,\ldots,J\}$ with birth rates $\lambda_0,\lambda_1,\ldots,\lambda_{J-1}$ and death rates $\mu_1,\mu_2,\ldots,\mu_J$.  Let $\{\pi_i\}$ denote the stationary distribution for $Q(\cdot)$ with cumulative distribution $P_i = \sum_{j=0}^i \pi_j$, and let
 $N_{\rm dep}(0,t]$ denote the number of deaths in $(0,t]$.
 Denote the departure rate by
\[
\lambda^* := \lim_{t \to \infty} \frac{\E(N_{\rm dep}(0,t])}{t} = \sum_{j=1}^J \mu_j \pi_j ,
\]
and write
\[
\Lambda^*_i= \frac{ \sum_{j=1}^i \mu_j \pi_j }{\lambda^*}
\]
for what are cumulative probabilities.
Then we know from \cite[Theorem 1]{NazarathyWeiss0336} that
\begin{equation}
\label{eq:newR1}
\RRR_{\pi} :=
\lim_{t \to \infty} \frac{\var\big(N_{\rm dep}(0,t]\big)}{\E(N_{\rm dep}(0,t])}
=
1 - 2 \sum_{i=0}^J (P_i - \Lambda^*_i)\Big(1 - \frac{\lambda^*}{\pi_i \lambda_i}  (P_i - \Lambda^*_i) \Big).
\end{equation}
Note that \eqref{eq:newR1} is in a slightly different form from that appearing in \cite{NazarathyWeiss0336}; the translation between the two forms is immediate.

In this paper we use  \eqref{eq:newR1} so as to obtain explicit expressions for $\RRR_{\beta,\eta}$ in the case of $M/M/s/K$ systems.  To do so we use the fact that in birth--death processes with $\lambda_i \equiv \lambda$, we have
\[
\lambda^*=\lambda (1- \pi_J), \qquad  \Lambda^*_i = \frac{P_{i-1}}{(1-\pi_J)},
\qquad
P_i - \Lambda^*_i = \frac{\pi_i - \pi_J P_i}{1-\pi_J},
\]
and after basic manipulation, \eqref{eq:newR1} can be represented as
\begin{equation}
\label{mainn}
\RRR_{\pi}
%= \RRR_{\beta, \eta}
= 1 - 2 \frac{\pi_J}{1-\pi_J} \sum_{i=0}^{J} P_i \Big(1 - \pi_J \frac{P_i}{\pi_i}\Big).
\end{equation}

In the case of the $M/M/s/K$ system, this elegant form proves amenable to manipulation and asymptotics under the QED regime, yielding our desired explicit formulae for $\RRR_{\beta,\eta}$ defined in \eqref{double}. Carrying out these asymptotics is the main contribution of the current paper. A further virtue of the form \eqref{mainn}, which is of independent interest, is that it demonstrates that (for the case $\lambda_i=\lambda$),
\begin{equation}
\label{eq:LowBoundHalf}
\RRR_\pi \ge \frac{1/2 - \pi_J}{1-\pi_J}.
\end{equation}
To see this observe that the function $x\mapsto x(1-x)$ is maximized at $x=1/4$, so that
\[
\frac{1}{2}(1-\RRR_\pi)(1-\pi_J) = \sum_{i=0}^J \pi_i   \frac{ \pi_JP_i}{\pi_i} \Big(1- \frac{ \pi_JP_i}{\pi_i}\Big)
\le \frac{1}{4}.
\]
The lower bound \eqref{eq:LowBoundHalf} implies that as long as $\pi_J \to 0$, as in the $M/M/s/K$ system, $\RRR_{\beta,\eta} \ge 1/2$. Our explicit expressions for $\RRR_{\beta,\eta}$ for $\rho \equiv 1$ in fact establish that $\RRR_{0,\eta}$ is in the range $(0.6, \frac{2}{3})$ with the exact value depending on $\eta$. Hence for QED systems, the magnitude of the BRAVO effect is not exactly the same as for single server systems, but it is in a similar range. Similar results are found for QED systems with non-zero $\beta$ for which $|\beta|$ is not too big.

%%%%%%%%%%%%%%%%%%%%%
%% PAPER STRUCTURE %%
%%%%%%%%%%%%%%%%%%%%%
The remainder of this  paper is structured as follows. In Section \ref{sec:bravoForQED} we give our main theorem on the BRAVO effect for the $M/M/s/K$ system, which presents an expression for the asymptotic ratio $\RRR_{\beta,\eta}$ for both the case $\rho \equiv 1$ (i.e.~$\beta=0$) and for the case $\beta \neq 0$ (i.e.~$\rho \approx 1-\beta/\sqrt{s}$).
The proofs for these two cases are presented in Sections \ref{proofthm1} and \ref{proofthm2}, respectively. We present some conclusions and ideas for future work in Section~\ref{sec:conclusion}. The appendix contains some needed asymptotic properties of Poisson probabilities.

%%%%%%%%%%%%%%%%%%%%%%%%%%%%%%%%%%%%%%%%%%%%%%%%%%
%%%%%%%%%%%%%%%%%%%%%%%%%%%%%%%%%%%%%%%%%%%%%%%%%%
%%%%%%%%%%%%%%%%%%%%%%%%%%%%%%%%%%%%%%%%%%%%%%%%%%
%%%%%%%%%%%%%%%%%%%%%%%%%%%%%%%%%%%%%%%%%%%%%%%%%%
%%%%%%%%%%%%%%%%%%%%%%%%%%%%%%%%%%%%%%%%%%%%%%%%%%
%%%%%%%%%%%%%%%%%%%%%%%%%%%%%%%%%%%%%%%%%%%%%%%%%%
%%%%%%%%%%%%%%%%%%%%%%%%%%%%%%%%%%%%%%%%%%%%%%%%%%
%%%%%%%%%%%%%%%%%%%%%%%%%%%%%%%%%%%%%%%%%%%%%%%%%%
%%%%%%%%%%%%%%%%%%%%%%%%%%%%%%%%%%%%%%%%%%%%%%%%%%
%%%%%%
%%%%%%%%%%%%%%%%%%%%%%%%%%%%%%%%%%%%%%%%%%%%%%%%%%
%%%%%%%%%%%%%%%%%%%%%%%%%%%%%%%%%%%%%%%%%%%%%%%%%%
%%%%%%%%%%%%%%%%%%%%%%%%%%%%%%%%%%%%%%%%%%%%%%%%%%
%%%%%%%%%%%%%%%%%%%%%%%%%%%%%%%%%%%%%%%%%%%%%%%%%%
%%%%%%%%%%%%%%%%%%%%%%%%%%%%%%%%%%%%%%%%%%%%%%%%%%
%%%%%%%%%%%%%%%%%%%%%%%%%%%%%%%%%%%%%%%%%%%%%%%%%%
%%%%%%%%%%%%%%%%%%%%%%%%%%%%%%%%%%%%%%%%%%%%%%%%%%
%%%%%%%%%%%%%%%%%%%%%%%%%%%%%%%%%%%%%%%%%%%%%%%%%%
%%%%%%%%%%%%%%%%%%%%%%%%%%%%%%%%%%%%%%%%%%%%%%%%%%
%%%%%%%%%%%%%%%%%%%%%%%%%%%%%%%%%%%%%%%%%%%%%%%%%%
\section{Main result}
\label{sec:bravoForQED}

We now state the main theorem of this paper, which identifies the BRAVO effect in many-server QED systems in terms of the asymptotic output ratio.
Let $\Phi$ and $\phi$ denote the distribution and density of a standard
normal random variable.

\begin{theorem}
%[case $\rho=1$]
\label{thm1}
Let $s,K\to\infty$ in such a way that ${K}/{\sqrt{s}}\rightarrow \eta$
for some finite positive $\eta$.

{\rm(a)} Let $\rho=1$. Then
\eqan{
\RRR_{0,\eta}=\frac23-L(\eta),
}
where
\eqan{L(\eta)=
\frac{\left(2-\frac{ \pi }{2}\right) \eta+ \sqrt{2 \pi } (1-\log 2-\frac{\pi}{12})}{ \left(\eta+\sqrt{\frac{\pi }{2}}\right)^3}.
}

{\rm(b)} Let $\rho=1-\beta/\sqrt{s}$ for finite $\beta\ne 0$.
Then
\eqan{\label{exacttt}
\RRR_{\beta,\eta}
= 1 - \frac{2\beta^2\e^{-\beta\eta}\h^2(\eta,\beta)}{\phi(\beta)}\f
   +g(\eta,\beta),
 }
where
\eqan{
\h(\eta,\beta)=\frac{1}{1-\e^{-\beta\eta}+\frac{\beta\Phi(\beta)}{\phi(\beta)}},
}
\eqan{
\f  = \int_{-\beta }^\infty
  \Big(1 - \beta \e^{-\beta\eta}\h(\eta,\beta) \frac{\Phi(-u)}{\phi(u)}\Big)
\,\Phi(-u)\,\d u,
 }
and, with $\h=\h(\eta,\beta)$,
\eqan{
  g(\eta,\beta)
=
2 \e^{-\beta\eta}\h(1+\e^{-\beta\eta}\h)\Big(1-\beta\eta-\e^{-\beta\eta}+(1-2\beta\eta\e^{-\beta\eta}-\e^{-2\beta\eta})\h\Big).
 }
\end{theorem}

We start by discussing Theorem \ref{thm1}(a). Figure~\ref{fig:rho1} displays $\RRR_{0,\eta}$ as a function of $\eta$.
Observe that as $\eta \to \infty$ which includes the case of a fixed finite number of servers (with a large, but finite buffer), we have $\RRR_{0,\eta} \to {2 \over 3}$. Further, for $\eta= 0$ we have
\eqan{
\RRR_{0,\eta} = 1-\frac{4(1-\log2)}{\pi} \approx 0.6093.
}
It is also easy to verify that as a function of $\eta$, $\RRR_{0,\eta}$ has a unique global minimum at
\eqan{
\eta =
\frac{
\sqrt{2 \pi}(\log8-2)
}{
4-\pi
}
\approx
0.232,
}
yielding $\inf _\eta \RRR_{0,\eta} \approx 0.6018$.
\begin{figure}
\begin{center}
\includegraphics[width=10cm]{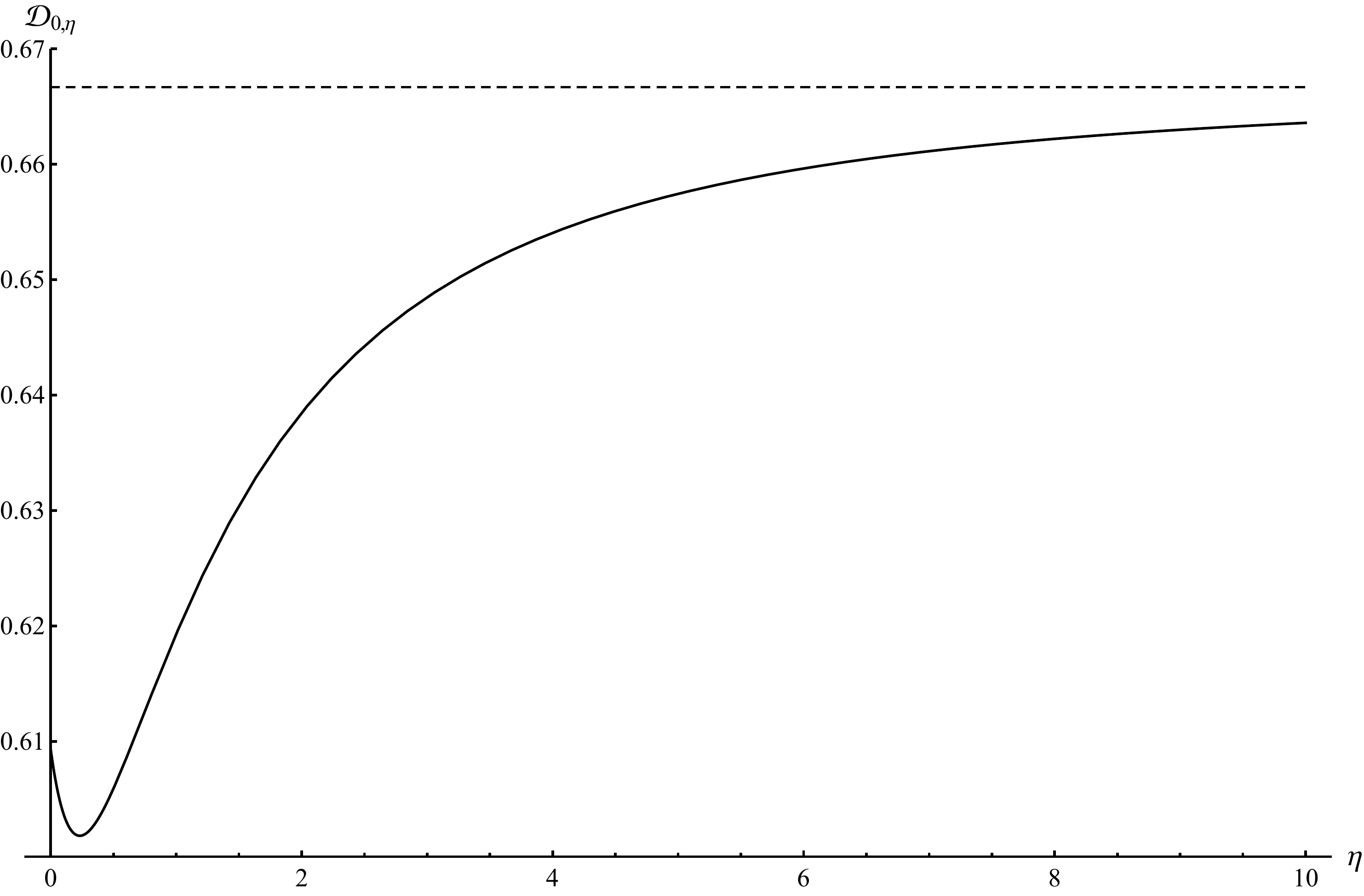}
\end{center}
\caption{{\small The asymptotic value of $\RRR_{0,\eta}$ as a function of $\eta$ when $\rho \equiv 1.$}
\label{fig:rho1}}
\end{figure}
We now turn to Theorem \ref{thm1}(b). Figure \ref{fig:rhon1} displays $\RRR_{\beta,\eta}$ as a function of $\beta$ and $\eta$. The figure suggests that, for fixed $\eta$,
\begin{equation*}
\lim_{\beta\to\infty}\RRR_{\beta,\eta}=\lim_{\beta\to -\infty}\RRR_{\beta,\eta}=1
\end{equation*}
and indeed this follows from the exact expression in \eqref{exacttt}. In fact, this can also be explained using the following heuristic reasoning.
As $\beta\to \infty$, the system becomes lightly loaded, and the process $Q_s(\cdot)$  behaves as an infinite-server system, which is reversible and therefore has an asymptotic output ratio equal to one. Also, as $\beta\to -\infty$, the system becomes increasingly overloaded, so that the process $Q_s(\cdot)$ behaves like a single-server system in which all servers work all the time. Such a single-server process is again a reversible birth--death process, which has an asymptotic output ratio equal to one. For any finite $\beta$, the behaviour of the process $Q_s(\cdot)$ resembles a mixture of an infinite-server system and a single-server system, and it is this alternation between two different stable systems that may explain the BRAVO effect. This effect is most pronounced for  values of $\beta$ close to zero.

\begin{figure}
\begin{center}
\includegraphics[width=12.5cm]{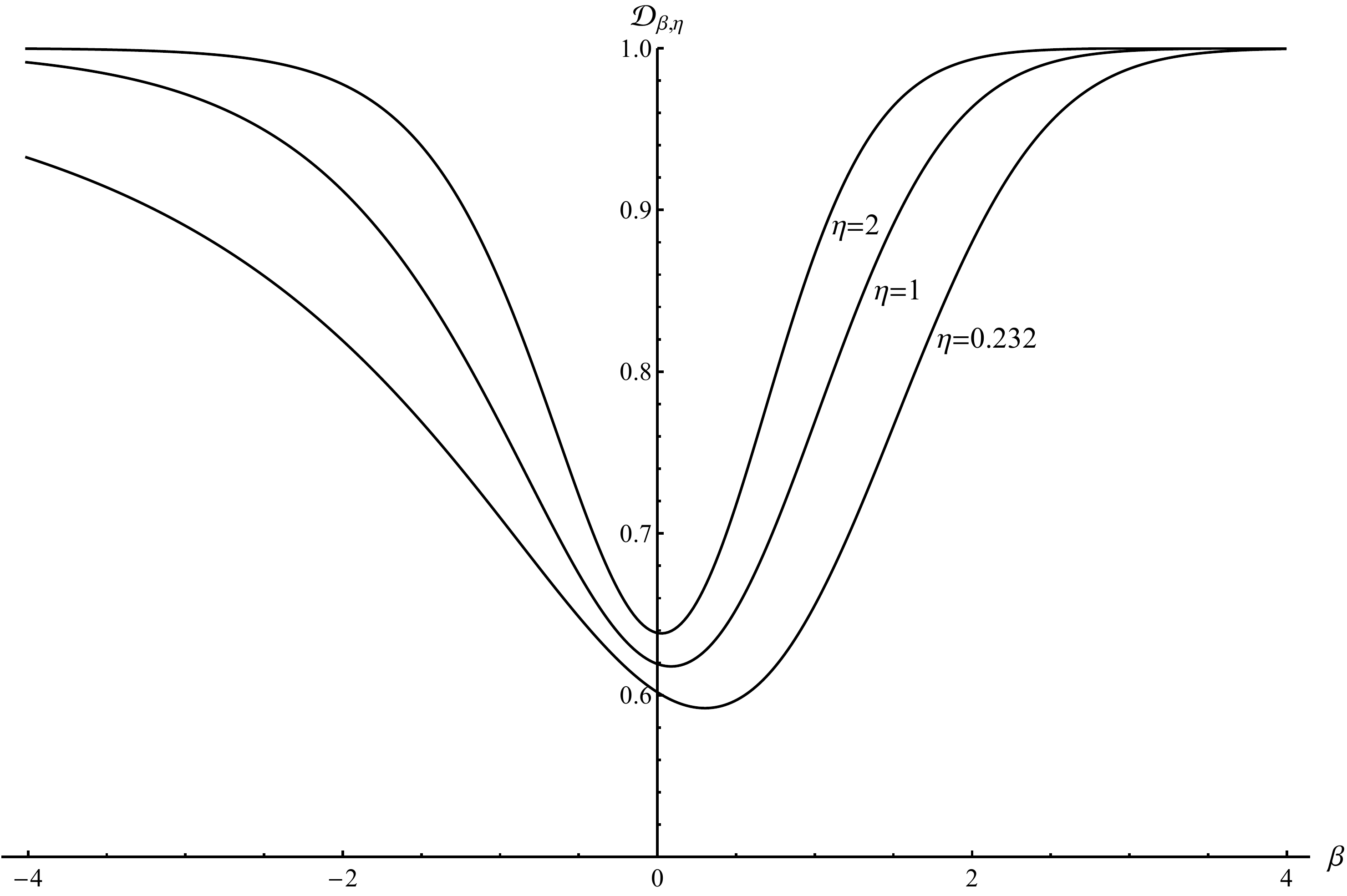}%3d.pdf
\end{center}
\caption{{\small The asymptotic value of $\RRR_{\beta,\eta}$ for various $\eta$ values.
\label{fig:rhon1}}}
\end{figure}

%Careful check for replacement
%\begin{itemize}
%\item $\eta=\sqrt{\pi/2}/\alpha$
%\item $\beta \h(\eta,\beta)=C_{\alpha,\beta}$
%\item $c_{\alpha,\beta}=1-\e^{-\beta\eta}$
%\item ...
%\end{itemize}

We now present some further arguments about why the scaling relations in \eqref{scaling} and \eqref{scaling2} are the precise scalings needed to create the BRAVO effect. We first show that the scaling relation in fact leads to a well-defined stochastic-process limit of the entire queue-length process. Then we establish, using Gaussian approximations for the Poisson distribution,
non-degenerate limits for the stationary distribution.

Let ``$\Rightarrow$'' denote weak convergence in the space $D[0,\infty)$ or convergence in distribution. The following result is proved in \cite[Thm.~1.2]{pang2007martingale}.

\begin{proposition}[Weak convergence to a diffusion process]\label{prop111}
Assume \eqref{scaling} and \eqref{scaling2}.  If $X_s(0)\Rightarrow X(0)\in\mathbb{R}$, then for every $t\geq 0$, as $s\to \infty$,
$
X_s(t)\Rightarrow X(t)\wedge\eta,
$
where the limit $X(\cdot)$ is the diffusion process with infinitesimal drift $ -\beta$ when $x> 0$ and $ -\beta-x$ when $x<0$, and constant infinitesimal variance $2$.
\end{proposition}

With minor abuse of notation, let $X$, $X_s$ and $Q_s$ without time arguments denote  stationary random variables.
One of the signature features of the QED regime is that, due to economies of scale, the stationary probability of delay $\mathbb{P}(Q_s\geq s)=\sum_{i=s}^{K}\pi_i$ converges to a limit that is neither zero nor one. This feature continues to exist for our model with the finite-capacity scaling. For convenience we state it formally below, showing the relation of $h(\eta,\beta)$ to the probability of delay (the result appears at equation (4.7) of the unpublished manuscript \cite{massey2005asymptotically}):

\begin{proposition}[Probability of delay in QED regime]\label{prop2}
Assume \eqref{scaling} and \eqref{scaling2}.  Then
\begin{equation}\label{qedresult}
\lim_{s\rightarrow\infty}\mathbb{P}(Q_s\geq s)=\frac{1-\e^{-\beta\eta}}{1-\e^{-\beta\eta}+ \frac{\beta\Phi(\beta)}{\phi(\beta)}}=(1-\e^{-\beta\eta})\h(\eta,\beta).
\end{equation}
\end{proposition}

It can be shown that the right-hand side of \eqref{qedresult} corresponds to the probability $\mathbb{P}(X>0)$ that the stationary diffusion process is positive (see \cite{browne1995piecewise}), as suggested by Proposition~\ref{prop111}. We choose, however, to give a direct derivation, starting from the exact expression for $\mathbb{P}(Q_s\geq s)$. We do so in order to give insight into the crucial role played by Gaussian approximations of the Poisson distribution; approximations constitute an important ingredient for the proof of the main theorem.

Recall the representation for the stationary distribution $\{\pi_i\}$
of the number in the many-server queueing system $M/M/s/K$, arrival rate $\lambda$ and service rate $\mu/s$ per
server,
\eqan{ \label{ssm}
\pi_i =
\begin{cases}
\frac{(s\lambda/\mu)^i}{i!}\,\pi_0 & \mbox{for} ~~ i=0,1, \ldots,s,\\
(\lambda/\mu)^{i-s}\pi_s & \mbox{for} ~~ i=s,s+1,\ldots,s+K,
\end{cases}
:=
\begin{cases}
\frac{{(s\rho)}^i}{i!}\,\pi_0,\\
\rho^{i-s} \pi_s.
\end{cases}
 }
For finite $\kappa>0$, let $\{\varpi_i(\kappa) = \e^{-\kappa}\kappa^i/i!\,, i=0,1,\ldots\}$ denote the Poisson distribution with mean $\kappa$. Observe that for $i=0,\ldots,s$,
\eqan{
\pi_i = \pi_0\e^{s\rho} \varpi_i(s \rho) := b_s(\rho)\varpi_i(s\rho),
 }
where $b_s(\rho) := \pi_0\e^{s\rho}$.

The stationary distribution $\{\pi_i\}$ has three parameters: $\rho$, $K$ and $s$, but they must satisfy the constraint
\eqan{
1 = \sum_{i=0}^{s+K} \pi_i
= \Big(\sum_{i=0}^s + \sum_{i=s+1}^{s+K} \Big)\pi_i.
 }
Thus when $\lambda=\mu$, i.e.\ $\rho=1$, the latter sum equals $K \pi_s$, and the former sum equals
\eqan{
b_s(1) \sum_{i=0}^s \varpi_i(s)
=
b_s(1)\,\Big(\frac12+\frac{\psi_{ss}}{\sqrt{s}}\Big)
 }
by the central limit property for the Poisson distribution (see Lemma~\ref{lemma:pois2}) for constants $\psi_{ss}$ satisfying $\sup_s |\psi_{ss}| < \infty$. Using Stirling's formula,
 \eqan{
 \frac{\pi_s}{\pi_0} = \frac{\varpi_s(s)}{\varpi_0(s)}
 = \frac{s^s}{s!}
= \frac{s^s}{(s/\e)^s \,\sqrt{2\pi s}\,}\Big(1+\frac{\vartheta_s}{12s}\Big)
= [1+O(s^{-1})]\e^s/\sqrt{2\pi s},
 }
for some $\vartheta_s$ for which $|\vartheta_s|<1$.
So $b_s(1) = \pi_0\e^s = \pi_s\sqrt{2\pi s}\,[1 + \vartheta_s/12s]$, and
\eqan{
1 = \pi_s\bigg[\Big(\frac12 + \frac{\psi_{ss}}{\sqrt{2\pi s}}\Big)
  \frac{\sqrt{2\pi s} }{1 + \vartheta_s/12s} + K\bigg].
  } Then
\eqan{
1 = \pi_s \sqrt{s}\Big(\sqrt{\pi/2} + \frac{K}{\sqrt{s}} + o(1)\Big),
 }
hence, using \eqref{scaling2},
\begin{equation}
\label{eq:pisRootS}
\pi_s\sqrt{s} \to \frac{1}{\sqrt{\pi/2}+\eta}.
\end{equation}
Also, $b_s(1) = \pi_s \sqrt{s}\,\sqrt{2\pi}\,(1+\vartheta_s/12s) \to \sqrt{2\pi}/(\sqrt{\pi/2}+\eta) =: b_\infty$ as $s \to \infty$.

More generally, for Theorem~\ref{thm1} we examine $M/M/s/K$ systems
for which $\rho_s = \lambda/\mu = 1 - \beta/\sqrt{s}$ and $s,K\to\infty$
as at \eqref{scaling2}.  The sums of terms $\pi_i$ over $i\le s$ and $i>s$ now equal
\eqan{
\label{eq:geoSum1}
b_s(\rho)        \sum_{i=0}^s \varpi_i(s\rho)
= b_s(\rho)       \Big[\Phi\Big(\frac{s-s\rho}{\sqrt{s\rho}\,}\Big)
+ O(s^{-\half})\Big]\quad
 \hbox{ and } \qquad \pi_s \sum_{i=1}^{K} \rho^i
  = \pi_s\frac{\rho(\rho^K -1)}{\rho-1}
 }
respectively. Here we have used  the central limit property of a Poisson distribution with mean $ s \rho$  as in Lemma~\ref{lemma:pois2}.
Substitution for $\rho$ gives, correct to terms that are $O(1/\sqrt{s})$,  $\pi_0\e^{s\rho} \Phi(\beta)$ and $\sqrt{s}\,\pi_s \big(1 - \e^{-\beta\eta}\,\big)/\beta$ when $s,K \to\infty$ as in \eqref{scaling2}.  The local asymptotic normality of Poisson probabilities (cf.~Lemma~\ref{lemma:pois1}) implies
\eqan{
b_s(\rho)       = \pi_s\cdot\frac{\pi_0}{\pi_s}\cdot \e^{s\rho}
\,=\, \frac{\pi_s}{\e^{-s\rho} (s\rho)^s\big/s!}
\,=\,   \frac{\sqrt{s\rho}\,\pi_s } {\phi\big( (s-s\rho)/\sqrt{s\rho}\,\big)}
\big(1+O(s^{-\half})\big)
\,=\,   \frac{\sqrt{s}\,\pi_s}{\phi(\beta)}\big(1+O(s^{-\half})\big).
 }
The terms on the left-hand sides of the relations in \eqref{eq:geoSum1} add to 1, so that
\eqan{\label{220}
1 + o(1) = \sqrt{s}\,\pi_s\Big(\frac{\Phi(\beta)}{\phi(\beta)}
+ \frac{1 - \e^{-\beta\eta} } {\beta}\Big)
 =: \frac{\sqrt{s}\,\pi_s} {\beta \h(\eta,\beta) }
 }
and $\sqrt{s}\,\pi_s \to \beta \h(\eta,\beta)$.   Also,
%$\lim_{\beta\to0} \beta \h(\eta,\beta) =
%1/(\thalf \sqrt{2\pi} + \eta) = , consistent with \eqref{eq:pisRootS}, and
$\lim_{s\to\infty} b_s(\rho) = \lim_{s\to\infty} \pi_0\e^{s\rho} = \beta \h(\eta,\beta)/\phi(\beta)$.
Proposition \ref{prop2} follows from combining \eqref{eq:geoSum1} and \eqref{220}.

\section{Proof of Theorem \ref{thm1}: the case $\rho=1$}\label{proofthm1}
 Since $\pi_i=\pi_s$ for $i\ge s$, $P_i = 1 - (K+s-i)\pi_s$. Write \eqref{mainn} as
\eqan{
\label{eq:newRdecompose}
(1-\pi_J)(1 - \RRR_{\beta,\eta}) =  2 \pi_s \sum_{i=0}^J \frac{\pi_J}{\pi_s}\,P_i
  \Big(1 - \frac{\pi_J}{\pi_i}\,P_i\Big) \,=\,
   2\pi_s \bigg(\sum_{i=0}^{s-1} + \sum_{i=s}^{s+K}\bigg)\big(\cdots\big),
 }
in which the last sum, for which $\pi_s=\pi_i=\pi_J$, equals
\begin{align*}
  2\pi_s \sum_{i'=0}^{K} \Big(1 - [K-i']\pi_s\Big)\,
   (K-i')\pi_s
 \,&=\, K(K+1)\pi^2_s - \frac{K(K+1)(2K+1)\pi_s^3}{3}\,,\\
 &\to
   \frac{\eta^2}{(\sqrt{\pi/2}+\eta)^2} - \frac{\frac{2}{3}\eta^3}{(\sqrt{\pi/2}+\eta)^3}
  \,=\, \frac{\eta^2 \sqrt{\pi/2}+\frac{1}{3}\eta^3}{(\sqrt{\pi/2}+\eta)^3}
\end{align*}
when $s,K\to\infty$ as at \eqref{scaling2}, using also \eqref{eq:pisRootS}.

It thus remains to consider the sum
\begin{equation}
\label{eq:mainSum1}
S_s := 2\sum_{i=0}^{s-1} \Big(1 - \frac{\pi_s}{\pi_i}\,P_i\Big)\pi_s P_i,
\end{equation}
which in terms of Poisson probabilities $\varpi_i := \varpi_i(s)$ and the
multiplier $b_s:=b_s(1)$ becomes
\begin{equation}
\label{eq:mainSum2}
S_s=
2 b_s^2
 \sum_{i=1}^s
\Big(1 - b_s\varpi_s \frac{\Pi_{s-i}(s)}{\varpi_{s-i}}\Big)
\varpi_s \Pi_{s-i}(s).
\end{equation}
Now both components of this sum for $S_s$ have finite limits (see below), and
$b_s\to b_\infty$ (see also below \eqref{eq:pisRootS}), finite and positive; we use this limit in examining the expression at \eqref{eq:mainSum2}.

Substitute from Lemmas \ref{lemma:pois1} and \ref{lemma:pois2} for the Poisson probabilities and assume for the moment that conditions for uniform convergence are met. Then
\eqan{
S_s
 \approx
2b_\infty^2 \sum_{i=1}^s \,\Big(1 - b_\infty \e^{\half x_{s i}^2} \Phi(-x_{s i})\Big) \frac{1}{\sqrt{2\pi}\,\sqrt{s}\,}
  \Phi(-x_{s i}),
 }
where $x_{si} = i/\sqrt{s}\,$. Recognize that this sum is an approximation to the Riemann integral
\begin{equation}
\label{eq:propInt1}
\frac{2b_\infty^2}{\sqrt{2\pi}} \int_0^{\sqrt{s}}
   \Big[1 - b_\infty\, \e^{\half u^2} \Phi(-u)\Big]
  \,\Phi(-u) \,\d u
\end{equation}
based on the $s$ intervals determined by the $s+1$ points $\{i/\sqrt{s}\,\}$, $i=0,\ldots,s$.  These integrals certainly converge as $s\to\infty$, because the (improper) Riemann integral is finite (its integrand is non-negative and dominated by $\Phi(-u)$ whose integral on $\Rpos$ equals $1/\sqrt{2\pi}\,$).

To prove that the improper integral based on \eqref{eq:propInt1} is indeed equal to  $\lim_{s\to\infty} S_s$, there are essentially two tasks: one relates to the finiteness of the limit of the finite sum, and the other to the convergence of the summands to the function involving the normal distribution function $\Phi$ and its density $\phi$.  These two tasks overlap in the sense that the part of the argument in the latter requiring uniform convergence follows from truncating the infinite sum.  It is convenient to rewrite the sum for $S_s$ as an integral, namely
\eqan{
\label{eq:mainInt1}
S_s =  \int_{\Rpos} g_s(x)\,\nu_s(\d x),
 }
where for each positive integer $s$, $\nu_s$ is a purely atomic measure on $\Rpos$ with support set $\{x_{si}\}:=\{i/\sqrt{s}:i=0,\ldots,s-1\}$ and mass $2 b_s^2 \varpi_s(s)\, \Pi_{s-i}(s)$ at $i/\sqrt{s}$, and $g_s(x)$ is a right-continuous simple function defined on $\Rpos$ equal to $1- b_s \Pi_s(s)$ at 0 and with upward jumps at each point $i/\sqrt{s}$, $i=1,\ldots,s$, where $g_s(x_{si}+0)= 1 -  b_s     (\varpi_s/\varpi_{s-i})\,\Pi_{s-i}(s)$.

We first show that $\Pi_{s-i}(s)/\varpi_{s-i}(s)$ is a monotone sequence (in $i$), and that its analogue $\e^{\half x^2} \Phi(-x)$ is monotonic in $x$.
For the former, recall that the partial sum $\Pi_i(s)$ is a tail integral of a gamma density function (e.g.\ Johnson {\sl et al.}\ (1993), equation (4.108)), so
\eqan{\label{4999}
\frac{\Pi_i(s)}{\varpi_i(s)}
  = \frac{\e^s i!}{s^i} \int_s^\infty \e^{-u} \frac{u^i}{i!} \,\d u
  \,=\, \int_0^\infty \e^{-v} \Big(1 + \frac{v}{s}\Big)^i\,\d v,
 }
which is monotonic increasing in $i$, hence $1 - b_s \Pi_s(s) \le 1 - b_s \varpi_s \Pi_{s-i}(s)/\varpi_{s-i}(s) \le 1 - b_s \varpi_s$ for $0\le i \le s$. This monotonicity implies that each function $g_s$ has all jumps upwards as asserted.

For the size of the jumps we evaluate (omitting here the common argument $s$)
\eqan{
\label{eq:omegaRatios}
\varpi_s\bigg[ \frac{\Pi_{i+1}}{\varpi_{i+1}}
   - \frac{\Pi_i}{\varpi_i}\bigg]
  = \varpi_s\,\frac{\varpi_i\big[\Pi_i + \varpi_{i+1}\big] - \varpi_{i+1}\Pi_i}
  {\varpi_{i+1}\,\varpi_i}
  \,=\, \varpi_s\Big[1 - \Big(1 - \frac{i+1}{s}\Big)\frac{\Pi_i}{\varpi_i}\Big],
 }
so the jumps in $g_s(x)$ are bounded uniformly in $x$ by $b_s \varpi_s(s) \le b_s  \big[1 + 1/12s\big]\big/\sqrt{2\pi\,s}\,\le B/\sqrt{s}$, for a constant $B$, uniformly in $s$.

For the putative limit function $\e^{\half x^2}\Phi(-x)$ in the integrand at \eqref{eq:propInt1}, we have for $x\in\Rpos$,
\eqan{
\sqrt{2\pi}\, \e^{\half x^2} \Phi(-x) = \int_x^\infty \e^{\half x^2} \e^{-\half v^2}\,\d v = \int_0^\infty \e^{-\half w^2 - w x}\,\d w\,,
 }
which is monotonic in $x$, so $1 - b_\infty \le 1 - b_\infty \e^{\half x^2} \Phi(-x) \le 1$ for $0\le x \to\infty$.  Thus, omitting the multiplier $b_\infty^2/\sqrt{\pi/2}$, the improper integral based on \eqref{eq:propInt1} is dominated by $\int_0^\infty \Phi(-u)\,\d u= 1/\sqrt{2\pi}\,,$ which is finite. Hence, given $\epsilon>0$, there exists $U_\epsilon$ such that
$$\int_0^{U_\epsilon} \Phi(-u)\big[1 - b_\infty \e^{\half u^2} \Phi(-u)\big]\,\d u$$
 is within $\epsilon$ of the similar integral but over $\Rpos$.
%[As an aside, Feller (1968), Lemma VII.1.2 notes the inequalities $1 - 1/x^2 \le x\,\Phi(-x)\big/\phi(x) \le 1$ for $ x>0$, else see Moran (1968, \S7.3) for a longer discussion of the Mills ratio.]

We turn to the measures $\nu_s$.  Consider the sums $S_{sj} := \varpi_s \sum_{i=0}^j  \Pi_i(s)$. Using again the integral representation for $\Pi_i(s)$ as in \eqref{4999},
\begin{align}
S_{si} &= \varpi_s \sum_{j=0}^i \int_s^\infty \e^{-u} \frac{u^j}{j!}\,\d u \nonumber
  \,=\,\varpi_s \int_s^\infty \d u \int_u^\infty \e^{-v} \frac{v^i}{i!}\,\d v
\\ \nonumber
&= \varpi_s \int_s^\infty \e^{-v}\frac{v^i}{i!}(v-s)\,\d v
= \varpi_s\big[(i+1)\Pi_{i+1}(s) - s \Pi_i(s)\big]
\\ \nonumber
&= \varpi_s \big[(i+1)\varpi_{i+1}(s) - (s - i-1)\Pi_i(s)\big]
= \varpi_s \big[s\varpi_i(s) - (s - i-1)\Pi_i(s)\big]
\\
&= s \,\varpi_s \varpi_i
\Big[1 - \Big(1 - \frac{i+1}{s}\Big)\frac{\Pi_i(s)}{\varpi_i}\Big]
\label{eq:brokenSum}
= s\varpi_s^2 \frac{\varpi_i}{\varpi_s}
\Big[1 - \Big(1 - \frac{i+1}{s}\Big)\frac{\Pi_i(s)}{\varpi_i}\Big]\,.
\end{align}
In particular, for $i=s-1$, $S_{s,s-1} = s\varpi_s\varpi_{s-1} = s\varpi_s^2$, ${}= \sum_{i=1}^s \varpi_s \Pi_{s-i}(s) = \nu_s(\Rpos)/2b_s^2$.

From Stirling's formula (see e.g.\ Appendix A) we know that $s\varpi_s^2 = [1 + \vartheta_s/12s]^2/2\pi$ for some $|\vartheta_s| \le 1$, so $\sup_s \nu_s(\Rpos)<\infty$ uniformly in $s$. Therefore, given $\epsilon > 0$, for every $s$ there exists $j_s$ such that $S_{s,j_s} > S_{s,s-1}-\epsilon$.

\endproof

\begin{proposition}
\label{prop:weakProp1}
The measures $\{\nu_s\}$ introduced at \eqref{eq:mainInt1}  converge weakly to the limit measure $\nu_\infty$ for which
\eqan{
\label{eq:nuProp}
\nu_\infty([0,x]) = \frac{2}{(\sqrt{\pi/2}+\eta)^2}\int_0^x \Phi(-v)\,\d v.
 }

\end{proposition}

\Remark: Because the limit function at \eqref{eq:nuProp} is continuous, $\sup_{x\in\Rpos} |\nu_s([0,x]) - \nu_\infty([0,x])| \,\to\,0$ for $s\to\infty$, i.e.\ we have not merely weak convergence but uniformity of that convergence on the domain of definition (e.g.\ Problem 5 in III.\S1 in \cite{shiryayev1984probability}).

\begin{proof} (of Proposition \ref{prop:weakProp1})
Let $f : \Rpos \mapsto \bbbR$ be a bounded continuous function; without loss of generality suppose $0 \le f(x)\le C_f$ for $x\in\Rpos$ and finite $C_f$. From our definitions,
\eqan{
\label{eq:intAndSum1}
\int_{\Rpos} f(x) \,\nu_s(\d x)  = 2b_s^2 \sum_{i=0}^s f(x_{si})\,\varpi_s \Pi_{s-i}(s),
 }
where $x_{si} = i/\sqrt{s}$ as earlier. Because $f$ is bounded and $S_{si} \le \nu_s(\Rpos)/2b_s^2 \to 1/(2\pi)$ as $s\to\infty$, there exists $\{j_s\}$ such that the integral at \eqref{eq:intAndSum1}, written as $\sum_{i=0}^{j_s-1} + \sum_{i=j_s}^s$, is approximated arbitrarily closely for $s$ large enough by the former of these sums; for example $j_s = s^{5/8}$.  Then using (A.6),
\begin{align}
\nonumber
\sum_{i=0}^{j_s-1} f(x_{si})\,\nu_s(\{i/\sqrt{s}\,\})
  &= 2b_s^2\sum_{i=0}^{j_s-1} f(x_{si})\,\varpi_s
  \Big[\Phi(-x_{si}) + \frac{\psi_{si}}{\sqrt{s}}\,\Big]
\\
\label{eq:bOmegaSum2}
&= 2b_s^2 \sum_{i=0}^{j_s-1} f(x_{si})\,\Phi(-x_{si}) \varpi_s
  + 2b_s^2 \sum_{i=0}^{j_s-1} f(x_{si})\,\frac{\psi_{si}}{\sqrt{s}} \,\varpi_s,
\end{align}
where the constants $\psi_{si}$ are bounded uniformly in $i$ and $s$, $|\psi_{si}| \le C_\psi$ say.
The last sum at \eqref{eq:bOmegaSum2} is bounded by $j_s C_f C_\psi/s \sqrt{2\pi}\,$, and when $j_s = s^{5/8}$, this bound is arbitrarily small for $s$ sufficiently large.  So the left-hand side of \eqref{eq:intAndSum1} is arbitrarily close to the sum at \eqref{eq:bOmegaSum2} involving $\Phi(\cdot)$.  Each of these sums is an approximation to the Riemann integral $2 b_\infty^2 \int_0^{s^{1/8}} f(v) \,\Phi(-v) \,\d v$, which in turn is arbitrarily close (for large enough $s$) to the improper integral over $\Rpos$.
\end{proof}

This weak convergence property of $\{\nu_s\}$ is not directly applicable to the integral at \eqref{eq:mainInt1} in which the functions $g_s$ are not continuous; but they are monotonic and bounded, and have bounded increments with $g_s(x+1/\sqrt{s})- g_s(x) \le B/\sqrt{s}$ (see around \eqref{eq:omegaRatios}).  We show that
 $g_\infty(x) := \lim_{s\to\infty}g_s(x)$ exists and equals the bounded continuous function $1 - b_\infty \e^{\half x^2} \Phi(-x)$ for $x\in\Rpos$. To this end, consider $x\in[0,U_\epsilon]$, where $U_\epsilon$ is as defined below \eqref{eq:omegaRatios}, and write
\eqan{
g_s(x) = 1 - b_s(1) \frac{\Phi\big(-i_s(x)\big/\sqrt{s}\,\big) +
    \psi_{s,i_s(x)}/\sqrt{s} }{\phi\big(i_s(x)\big/\sqrt{s}\,\big)
  \big[1 + O(s^{-1/8  })\big]\big/\phi(0)}\,,
 }
where $i_s(x) = \min\big(\lfloor x\sqrt{s}\,\rfloor,\,s\big)$. For given $x$, $i_s(x)/\sqrt{s} \to x$ as $s\to\infty$, and thus $g_s(x) \to g_\infty(x)$ as defined.  Further, because $\phi(i_s(x)/\sqrt{s}\,)$ is bounded away from 0 on $[0,U_\epsilon]$, this convergence satisfies
$\sup_{x\in [0,U_\epsilon]} |g_s(x) - g_\infty(x)| < B_g/\sqrt{s}$ for
some finite constant $B_g$.  We can now write
\begin{align}
 \int_{\Rpos} g_s(x)\,\nu_s(\d x)
 &= \int_{(\Ueps,\infty)} g_s(x) \nu_s(\d x)
   + \int_{[0,\Ueps]} [g_s(x) - g_{\infty(x)}]\,\nu_s(\d x) \nonumber
 \\
&\qquad+ \int_{[0,\Ueps]} g_\infty(x)\,[\nu_s(\d x) - \nu_\infty(\d x)]
   - \int_{(\Ueps,\infty)} g_\infty(x)\,\nu_\infty(\d x) \nonumber
\\ \nonumber
&\qquad+ \int_{\Rpos} g_{\infty}(x)\,\nu_\infty(\d x)
\\
\label{eq:5terms}
&:= T_1 + T_2 + T_3 - T_4 + \int_{\Rpos} g_{\infty}(x)\,\nu_{\infty}(\d x).
\end{align}
Each of the terms $T_1,\ldots,T_4$ here can be made smaller than any given positive $\eps$, first by the choice of $\Ueps$ for $T_4$, then by choice of $s$ for $T_3$ and $T_1$ by appealing to weak convergence, and finally by choice of $s$ for $T_2$ by the uniform convergence of $g_s(x)$ in $[0,\Ueps]$. \endproof

It remains to demonstrate that the improper integral in \eqref{eq:propInt1} has the value as asserted.  We do this
in the proof of the next lemma.
\begin{lemma}\label{lemm6}
\[
I_+:=\int_0^\infty \Phi(-u)[1 - \e^{\half u^2} \Phi(-u)]\,\d u = \sqrt{2 \pi} (1 - \log \sqrt{2}).
\]
\end{lemma}
\begin{proof}
Write $\int_0^\infty \Phi(-u)\,\d u - I_+
  = \int_0^\infty \e^{\half u^2} \,\d u (2\pi)^{-1}
  \int_u^\infty \e^{-\half v^2} \,\d v \int_u^\infty \e^{-\half w^2}\,\d w$,
noting that the left-hand side equals $1/\sqrt{2\pi} - I_+ =: J_0$ say.
Use polar coordinates
 $(v,w) = (r\cos\theta, r\sin\theta)$ so that
$ \d v\,\d w \mapsto r\,\d r \,\d\theta $ and multiply by $2\pi$, i.e.\
\begin{align*}
2\pi J_0 :\!&=
\int_0^\infty \e^{\half u^2}\,\d u
  \int_u^\infty \e^{-\half v^2}\,\d v
  \int_u^\infty \e^{-\half w^2}\,\d w \\
&=
\int_0^\infty \e^{\half u^2}\,\d u
  \int_{u\sqrt{2}}^\infty \Big(\thalf\pi - 2 \arcsin \frac{u}{r}\Big)
  \e^{-\half r^2} r\,\d r \\
&=
0 + 2 \int_0^\infty \e^{\half u^2}\,\d u
  \int_{u\sqrt{2}}^\infty
\Big(- \frac{u}{r^2}\cdot \frac{1}{\sqrt{1 - u^2/r^2}}  \Big)
\big( - \e^{-\half r^2} \big)\,\d r \qquad\hbox{(integration by parts)},\\
&=
2 \int_0^\infty \e^{-\half z^2 \cos 2\theta } z\,\d z
  \int_0^{\artan(1/\sqrt{2}\,)}  \frac{ z \sin\theta }{ z \cos\theta
  \sqrt{z^2 \cos 2\theta} }\,\d\theta
\qquad\big((r,u)=(z\cos\theta,z\sin\theta)\big),\\
&=
2 \int_0^{\artan(1/\sqrt{2}\,)}
  \frac{\tan\theta\,\d\theta}{\sqrt{\cos 2\theta}}
  \int_0^\infty \e^{-\half z^2 \cos 2\theta}\,\d z
   \qquad\hbox{(Fubini's theorem)},\\
  &= 2 \int_0^{\artan(1/\sqrt{2}\,)}
  \frac{\tan\theta\,\d\theta}{\cos 2\theta}
  \int_0^\infty \e^{-\half y^2}\,\d y
  \qquad(y = z\sqrt{\cos 2\theta}\,),\\
&=
\sqrt{2\pi}\int_0^{\artan(1/\sqrt{2}\,)}
  \frac{\tan \theta \,\d\theta}{\cos 2\theta}
  \qquad\hbox{(integration)},\\
&=
\sqrt{2\pi} \int_0^{1/\sqrt{2}} \frac{ t \,\d t}{1-t^2}
  \qquad(t=\tan\theta),\\
&=
\, \sqrt{2\pi} \int_0^\half
  \frac{\half \,\d u}{1-u}
 \,=\, \thalf \sqrt{2\pi} \log \big(1/\thalf\big) \ =  \
\sqrt{2\pi}\,\log\sqrt{2}\,.
\end{align*}
Recalling the definition of $J_0$ yields the expression for $I_+$.
\end{proof}

%%%%%%%%%%%%%%%%%%%%%%%%
%%%%%%%%%%%%%%%%%%%%%%%%
%%%%%%%%%%%%%%%%%%%%%%%%
%%%%%%%%%%%%%%%%%%%%%%%%
%%%%%%%%%%%%%%%%%%%%%%%%
%%%%%%%%%%%%%%%%%%%%%%%%
%%%%%%%%%%%%%%%%%%%%%%%%
%%%%%%%%%%%%%%%%%%%%%%%%
%%%%%%%%%%%%%%%%%%%%%%%%
\section{Proof of Theorem \ref{thm1}: the case $\rho=1-\beta/\sqrt{s}$}
\label{proofthm2}

We now have $\rho=1-\beta/\sqrt{s}$ for finite $\beta\ne 0$. The proof is similar to the approach and methods used in the case $\rho=1$ above; the particular algebraic expressions differ.  Start by breaking the sum at \eqref{eq:newRdecompose} into the two sums as before. Write the first of these as
\eqan{
\label{eq:sumInt2}
2 \sum_{i=1}^s
   \Big(1 - \frac{\pi_J}{\pi_s}\,\varpi_s(s\rho)\,
   \frac{\Pi_{s-i}(s\rho)}{\varpi_{s-i}(s\rho)}\Big)
\pi_J \Pi_{s-i}(s \rho) \,=\, \int_{\Rpos} g_s(x)\,\nu_s(\d x),
 }
where the measure $\nu_s$ is atomic as before but now has mass $ 2b^2_s(\rho) \varpi_s(s\rho)\,\rho^K \Pi_{s-i}(s\rho)$ located at $x_{si}=i/\sqrt{s}$ for $i=0,\ldots,s-1$, and the function $g_s$ is a right-continuous simple function defined on $\Rpos$ equal to $1- b_s(\rho)\rho^K \Pi_s(s\rho)$ at 0 and has jumps at each point $x_{si}$, $(i=1,\ldots,s)$, where $g_s(x_{si}+0)=1 - b_s(\rho)\rho^K \Pi_{s-i}(s\rho)\big/ [\varpi_{s-i}(s\rho)/\varpi_s(s\rho)]$,  so that $g_s(\cdot)$ is monotonic.

To check that $\nu_s$ has finite total mass, and that there is a bound on $\nu_s(\Rpos)$ that is uniform in $s$, note that $\lim_{s,K\to\infty} b_s^2(\rho) \rho^K$ is finite for finite $\beta$ and equal to $[\beta \h(\eta,\beta)/\phi(\beta)]^2(1-\e^{-\beta\eta})$ (see around (2.7)); for the rest, mimic the
calculation leading to \eqref{eq:brokenSum}, now with $\rho=1-\beta/\sqrt{s}\,,$ in computing
\begin{align*}
S_{si} := \varpi_s(s\rho) \sum_{j=0}^i \Pi_j(s\rho)=s\rho\,\varpi_s(s\rho) \varpi_i(s\rho)
\bigg[1 - \Big(1 - \frac{i+1}{s\rho}\Big)
   \frac{\Pi_i(s\rho)}{\varpi_i(s\rho)}\bigg]
   \end{align*}
so that
\begin{align*}
 S_{s,s-1} = s\rho\,\varpi_s(s\rho)\,\varpi_{s-1}(s\rho)
\bigg[1 - \Big(1 - \frac{1}{\rho}\Big)
   \frac{\Pi_{s-1}(s\rho)}{\varpi_{s-1}(s\rho)}\bigg]
= s\, \varpi_s^2(s\rho)\bigg[1 + \frac{\beta }{\rho\sqrt{s}\,}
   \frac{\Pi_{s-1}(s\rho)}{\varpi_{s-1}(s\rho)}\bigg],
\end{align*}
which for given (finite) $\beta$ is bounded in $s$.  Thus, $\sup_s \nu_s(\Rpos) <\infty$.  Note that
\eqan{
\nu_s(\Rpos) = 2 b^2_s(\rho) \rho^K S_{s,s-1}  \,\to\, \frac{\beta^2\h^2(\eta,\beta)\,(1-\e^{-\beta\eta})}{\pi\,\phi^2(\beta)}\, \big(1 + \beta\sqrt{\pi/2}\,\big) \qquad (s\to\infty).
 }

From Lemma~\ref{lemma:pois2}, for $\rho = 1 - \beta/\sqrt{s}$,  $0\le i\le s$ and $s$ large,
\eqan{
\Pi_{s-i}(s\rho)
  - \frac{\psi_{si}}{\sqrt{s\rho}\,}
= \Phi\Big(\frac{s-i-s\rho}{\sqrt{s\rho\,}}\Big)
 &= \Phi\big(- x_{si} + \beta +O(s^{-\half})\big)\nonumber\\
 &= \Phi\big(-x_{si} + \beta
\big) +O(s^{-\half}),
 }
where $o(\cdot)$ in the last term is uniform in $i$ and $s$. Hence
\eqan{
\label{eq:2ndPoissRel}
\Pi_{s-i}(s\rho) =
  \Phi\big(- x_{si} + \beta\big)
  + \tilde\psi_{si}\big/\sqrt{s}\,,
 }
where by Lemma~\ref{lemma:pois2} $\tilde\psi_{si}$ is uniformly bounded in $i$ and $s$.

For $x\in\Rpos$ define $i_s(x) = \lfloor x\sqrt{s} \rfloor$ so that $i_s(x)/\sqrt{s} \to x$ as $s\to\infty$, and
\begin{align*}
g_s(x) &= 1 -  \frac{b_s(\rho)\rho^K\,\Pi_{s-i_s(x)}(s\rho)}
 {\varpi_{s-i_s(x)}(s\rho)\big/ \varpi_s(s\rho)}
 \\
  &= 1 - b_s(\rho)
\frac{\rho^K\big[\Phi(-x_{s,i_s(x)}+\beta)
  + \tilde\psi_{s,i_s(x)}/\sqrt{s}\,\big] }
  {\phi(-x_{s,i_s(x)} + \beta)/\phi(\beta)\,[1 + O(s^{-\half})]}\,.
\end{align*}
Then for given $x$ and $s\to\infty$, the right-hand side $\to g_\infty(x)$ defined by
\eqan{\label{45}
g_\infty(x) := 1 - \beta \h(\eta,\beta)\,\frac{(1-\e^{-\beta\eta}) \Phi(-x+\beta)}
   {\phi(-x+\beta)}
 }
at a uniform rate of convergence that is $O(s^{-\half})$ for $x$ on compact sets like $\Ueps$ where the denominator above is bounded away from 0, $b_\infty(\rho) := \lim_{s\to\infty} b_s(\rho) = \beta \h(\eta,\beta)/\phi(\beta)$ as below (2.7), and
\eqan{
\lim_{s\to\infty}
   \big(1 - \beta/\sqrt{s}\,\big)^{\sqrt{s}\eta}
  = \e^{-\beta\eta}.
 }

\begin{proposition}
\label{prop:weakProp2}
The measures $\{\nu_s\}$ defined below \eqref{eq:sumInt2} converge weakly to the limit measure $\nu_\infty$ for which
\eqan{
\label{eq:nuProp2}
\nu_\infty([0,x]) =
  \frac{2 [\beta \h(\eta,\beta)]^2 \,(1-\e^{-\beta\eta})}{\phi(\beta)}
\int_{(-\beta,x-\beta]} \Phi(-v)\,\d v.
 }
\end{proposition}

\begin{proof}
The same argument as used in establishing Proposition~\ref{prop:weakProp1} holds,  subject to using \eqref{eq:2ndPoissRel} in place of \eqref{eq:piRel3} from Lemma~\ref{lemma:pois2}. The analogue of \eqref{eq:bOmegaSum2} for $\beta\ne 0$, for bounded continuous $f$, is
\eqan{
\sum_{i=0}^{j_s-1} f(x_{si}) \,\nu_s\big(\{i/\sqrt{s}\,\}\big)
   = \sum_{i=0}^{j_s-1} f(x_{si})\,2 b^2_s(\rho)\rho^K \varpi_s(s\rho)
  \Big[\Phi(-x_{si} + \beta) + \frac{\tilde\psi_{si}}{\sqrt{s}\,}\Big].
 }
Each term involving some $\tilde\psi_{si}$ is at most $O(1/s)$, so when $j_s< O(s)$ their sum is negligibly small for large $s$. Since $\varpi_s(s\rho) =       \e^{-\half\beta^2}/\sqrt{2\pi s}\,[1 + O(s^{-\half})]$ the other term is an approximation to the Riemann integral
\eqan{
2\big(\lim_{s\to\infty} b^2_s(\rho) \rho^K\big)\phi(\beta)
\int_0^{j_s/\sqrt{s}}  f(u)\,\Phi(-u+\beta)\,\d u,
 }
from which we deduce \eqref{eq:nuProp2}.
\end{proof}

An argument similar to that leading to \eqref{eq:5terms} shows that the integrals at the right of \eqref{eq:sumInt2} converge to the limit $\int_{\Rpos} g_\infty(x)\, \nu_\infty(\d x)$.

It remains to consider the analogue of the last sum at \eqref{eq:newRdecompose} for which $i\ge s$.  As earlier, calculate $P_i^{}$ via its tail which now reads
\[
P_i^{} = 1 - \sum_{j=i+1}^J \pi_j
  \,=\, 1 - \pi_J\sum_{j=i+1}^J \rho^{j-J}
 \,=\, 1 - \rho^{i+1-J} \frac{\pi_J(\rho^{J-i}  - 1)}{\rho-1}
 = 1 + \frac{\rho\pi_J(1 - \rho^{-(J-i)})}{1-\rho}\,,
\]
so
\begin{align}
\Big(\,1 - \frac{\pi_J}{\pi_i}P_i^{}\Big)P_i
 &= \Big(1 - \rho^{J-i} P_i^{}\Big)P_i \nonumber
\\
&= \Big(1 + \frac{\rho\pi_J(1 - \rho^{-(J-i)})} {1 - \rho}\Big)
   \bigg[1 - \rho^{J-i} \Big(1 + \frac{\rho\pi_J(1 - \rho^{-(J-i)})}{1 - \rho}
\Big)\bigg] \nonumber
\\
&= \Big(1 + \frac{\rho \pi_J(1-\rho^{-(J-i)}}{1-\rho}\Big)(1-\rho^{J-i})
  \Big(1 + \frac{\rho\pi_J}{1-\rho}\Big) \nonumber
\\
&= \Big(1 + \frac{\rho\pi_J}{1-\rho}\Big)
  \bigg[1 - \rho^{J-i} + \frac{\rho\pi_J\big(2 - \rho^{J-i} - \rho^{-(J-i)}\big)}
    {1-\rho} \bigg]\,.
\end{align}
Hence, for the second part of the sum of \eqref{eq:newRdecompose} for the case $\rho\neq 1$ and $i\geq s$ we have
\begin{align}
2\pi_J\sum_{i=s}^J \Big(1 - \frac{\pi_J}{\pi_i}P_i^{}\Big)P_i \nonumber
  &=  2\pi_J\Big(1 + \frac{\rho\pi_J}{1-\rho}\Big)
  \bigg[(K+1)\Big(1 + \frac{2\rho\pi_J}{1-\rho}\Big)
\\
  &
  - \frac{\rho^K (\rho^{-K-1}-1)} {\rho^{-1}-1} \nonumber
\Big(1 + \frac{\rho\pi_J}{1-\rho}\Big)
  - \frac{\rho^{-K} (\rho^{K+1} - 1)}{\rho-1} \cdot\frac{\rho\pi_J}{1-\rho}
\bigg]
\\
\label{eq:bigSecondSum}
=  2\pi_J\Big(1 + \frac{\rho\pi_J}{1-\rho}\Big)
  &\bigg[\Big(K+1-\frac{1-\rho^{K+1}}{1-\rho}\Big)
     \Big(1 + \frac{\rho\pi_J}{1-\rho}\Big)
  + \Big(K+1 - \frac{1-\rho^{K+1}}{\rho^K(1-\rho)}\Big)
 \frac{\rho\pi_J}{1-\rho}
\bigg].\qquad\quad
 \end{align}
 For $\rho = 1 - \beta/\sqrt{s}$ and $s,K\to\infty$ as at \eqref{scaling2}
\eqan{
\frac{\rho \pi_J}{1-\rho} = \frac{\tsqrt{s}\,\pi_s \rho^{K+1}}{\beta}
\approx \frac{\sqrt{s}\,\pi_s (1-\e^{-\beta\eta}) }{\beta}\,,
 }
so the limit of the right-hand side of \eqref{eq:bigSecondSum} is the same as the limit of
\begin{align}
&-2\pi_s \rho^K\Big(1 + \frac{\sqrt{s}\,\pi_s\,(1-\e^{-\beta\eta})}{\beta}\Big)
  \bigg[ \Big(K - \frac{\sqrt{s}(1 - (1-\e^{-\beta\eta}))}{\beta}\Big)
\Big(1 + \frac{\sqrt{s}\,\pi_s\,(1-\e^{-\beta\eta})}{\beta}\Big)
\nonumber \\ \nonumber
&\quad + \Big(K - \frac{\sqrt{s}(1 - (1-\e^{-\beta\eta}))}{(1-\e^{-\beta\eta})\,\beta} \Big)
  \frac{\sqrt{s}\,\pi_s\,(1-\e^{-\beta\eta})}{\beta}
\bigg]
\\ \nonumber
 &\to -2\beta \h(\eta,\beta)\,(1-\e^{-\beta\eta})
  \Big(1 + \frac{\beta \h(\eta,\beta)\,(1-\e^{-\beta\eta})}{\beta}\Big)
  \bigg[\eta - \frac{1-(1-\e^{-\beta\eta})}{\beta}\Big)
   \Big(1 + \frac{\beta \h(\eta,\beta)\,(1-\e^{-\beta\eta})}{\beta}\Big)
   \\ \nonumber
 &\quad + \Big(\eta
  - \frac{1-(1-\e^{-\beta\eta})}{\beta\,(1-\e^{-\beta\eta})}\Big)
  \frac{\beta \h(\eta,\beta)\,(1-\e^{-\beta\eta})}{\beta}\bigg]
\\
 &= -2\tC\bigg(1 + \frac{\tC}{\beta}\bigg) \bigg[
 \bigg(\eta - \frac{1-(1-\e^{-\beta\eta})}{\beta}\bigg)
  \bigg(1 + \frac{\tC}{\beta}\bigg)
 +\bigg(\eta
 - \frac{1-(1-\e^{-\beta\eta})}{\beta (1-\e^{-\beta\eta})}\bigg)
  \frac{\tC}{\beta}\bigg],\label{413}
\end{align}
where $\tC=\beta \h(\eta,\beta)\,(1-\e^{-\beta\eta})$. Combining \eqref{45} and \eqref{eq:nuProp2}, and \eqref{413}, as required for the sum at \eqref{eq:newRdecompose}, now yields
\begin{equation}\label{intttt}
1-\RRR_{\beta,\eta}
=  2\lim_{s\to\infty}\tsqrt{s}\,\pi_s \int_{-\beta}^\infty
   (1-\e^{-\beta\eta})
  \Big(1 - \frac{(1-\e^{-\beta\eta})\,\phi(\beta)\,\Phi(-u)}{\phi(u)}
   \Big)\Phi(-u)\,\d u + g(\eta,\beta).
 \end{equation}
We should like at this point to evaluate the integral at \eqref{intttt}.
For this purpose we have mimicked the steps followed in establishing the
expression for $I_+$ in Lemma~\ref{lemm6}, but have succeeded only
in reaching the expressions in \eqref{111} and \eqref{222} below in which final
quadratures are unresolved.  Inspection of \eqref{intttt} shows that we must evaluate,
for both positive and negative $\beta$,
 $$ J_\beta :=  \int_{\beta}^\infty \frac{[\Phi(-u)]^2}{\phi(u)}
  \,\d u \,=\, \int_{\beta}^\infty \e^{\half u^2}\,\d u
  \int_u^\infty \int_u^\infty \frac{\e^{-\half(v^2+w^2)}}{2\pi}\,\d v \,\d w.
$$
By following steps similar to those in the proof of Lemma~\ref{lemm6} we deduced first
that for $\beta>0$,
\eqan{\label{111}J_\beta  =
- \frac{\Phi(-|\beta|\,)}{\sqrt{2\pi}}\,\log 2
  + \int_1^\infty \frac{\beta \e^{-\half \beta^2 x^2}}{2\pi} \,
   \log\Big(1 + \frac{1}{x^2}\Big)\,\d x.
   }
For $\beta<0$, integration takes place (cf.\ second line of the evaluation of
$J_0$ in the proof of Lemma~\ref{lemm6}) over $0<r<\infty$, with angular coordinate having arc-lengths $2\pi$
for $r<|\beta|$, $2\pi - 4\,\arccos(|\beta|/r) = 4\,\arcsin(|\beta|/r)$ for
$|\beta|<r<|\beta|\sqrt{2}$, and $\half\pi + 2\,\arcsin(|\beta|/r)$
for $r>|\beta|\sqrt{2}$.  Ultimately this yields
\begin{align}\label{222}
J_{-|\beta|}
  = J_{|\beta|} - J_0 + \int_0^{|\beta|} \e^{\half u^2}\,\d u
  &+ \frac{4\log 2}{\sqrt{2\pi}}\Big[\Phi(|\beta|\sqrt{2}\,) - \half\Big]\cr
  &- \int_0^{\sqrt{2}} \frac{4|\beta|\e^{-\half\beta^2x^2}}{2\pi}
\,\log\Big(1 + \frac{1}{x^2}\Big)\,\d x.
\end{align}

%%%%%%%%%%%%%%%%%%%%%%%%%%%%%%%%%%%%%%%%%%%%%%%%%%
%%%%%%%%%%%%%%%%%%%%%%%%%%%%%%%%%%%%%%%%%%%%%%%%%%
%%%%%%%%%%%%%%%%%%%%%%%%%%%%%%%%%%%%%%%%%%%%%%%%%%
%%%%%%%%%%%%%%%%%%%%%%%%%%%%%%%%%%%%%%%%%%%%%%%%%%
%%%%%%%%%%%%%%%%%%%%%%%%%%%%%%%%%%%%%%%%%%%%%%%%%%
%%%%%%%%%%%%%%%%%%%%%%%%%%%%%%%%%%%%%%%%%%%%%%%%%%
%%%%%%%%%%%%%%%%%%%%%%%%%%%%%%%%%%%%%%%%%%%%%%%%%%
%%%%%%%%%%%%%%%%%%%%%%%%%%%%%%%%%%%%%%%%%%%%%%%%%%
%%%%%%%%%%%%%%%%%%%%%%%%%%%%%%%%%%%%%%%%%%%%%%%%%%

\section{Conclusion}
\label{sec:conclusion}

The analysis presented in this paper both answers questions and raises
further points that need clarification. Recall that in the setup in \cite{NazarathyWeiss0336}, Nazarathy and Weiss observed that for $M/M/1/K$ systems with $K \to \infty$, $\RRR_\pi \to 1$ for $0<\rho<\infty$ except when $\rho=1$, for which they found
$\RRR_\pi \to {2\over 3}$.  A similar effect occurs in multi-server $M/M/s/K$ systems. This is illustrated in Figure~\ref{fig:preLim}.

\begin{figure}[h]
\begin{center}
\includegraphics[width=12cm]{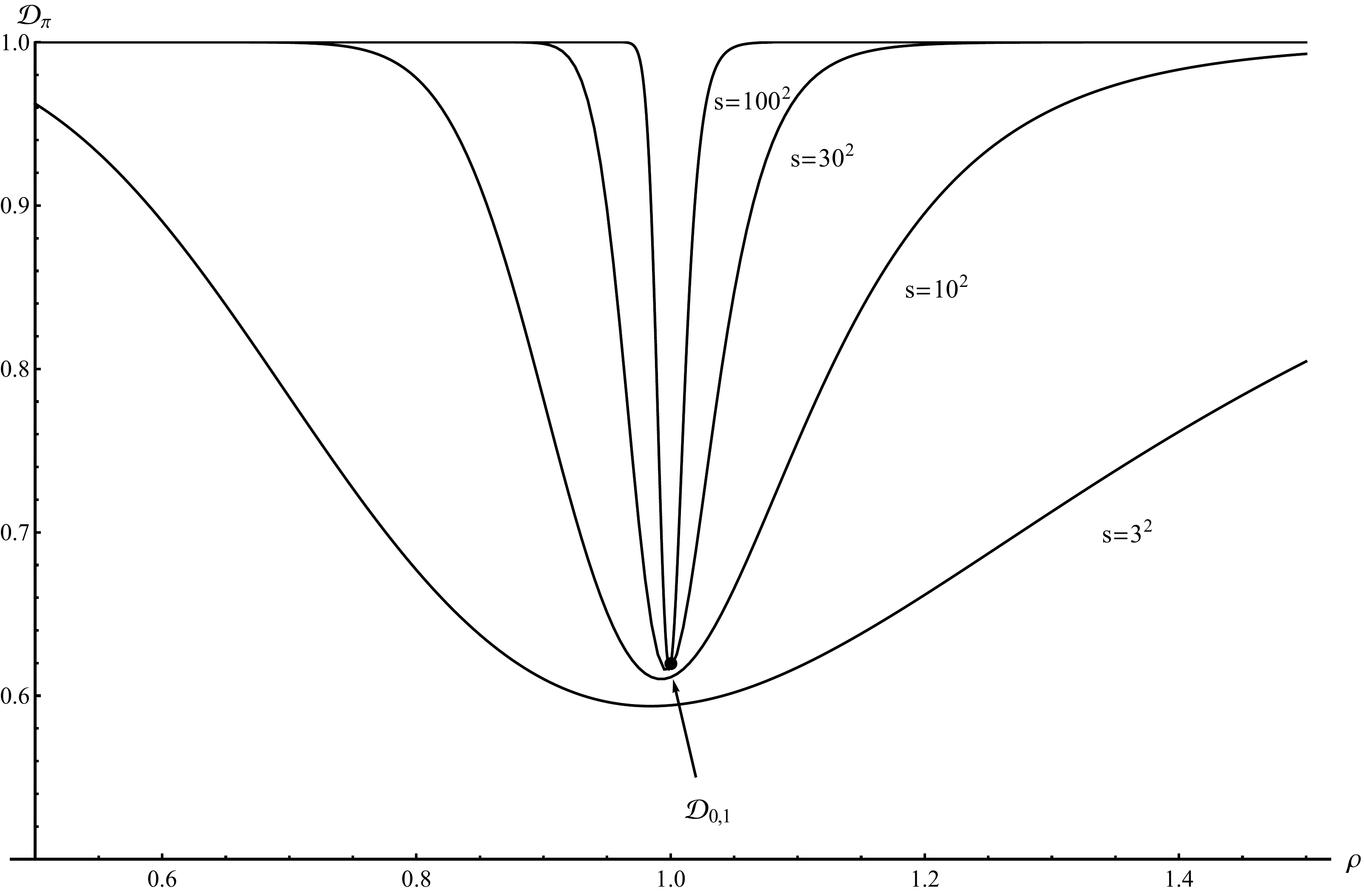}
\end{center}
\caption{{\small M/M/$s/K$ systems with $K=\sqrt{s}$ for increasing values of $s$. $\RRR_\pi$ as a function of $\rho$. The black dot is $\RRR_{0,1}$ corresponding to $s \to \infty$ with $\eta=1$.
\label{fig:preLim}}}
\end{figure}

Our theorem shows that in the multi-server case, this discontinuity is merely a matter
of scale: setting $\beta = (1-\rho)\sqrt{s}$ and considering
 $$
 %\RRR=
 \RRR_\beta =
 \lim_{t\to\infty}
 \frac{\var\big(N_\dep(0,t]\big)}{ \E \big( N_\dep(0,t]\big)}
\Bigg|_{\rho=1-\beta/\sqrt{s}}
\quad ,$$
the limit function  is no longer discontinuous but reflects the
asymptotic variance over a much finer range of values as in Figure~\ref{fig:rhon1} of Section~\ref{sec:bravoForQED}.

In turn, this emphasizes that, for the systems that we have considered, there
exists some
subtle interplay between the service facility on the one hand and customers
on the other. In terms of customers there may be a `deficit' (so some servers are idle)
or a `surfeit' (so some customers are waiting), and when
the system is `balanced', meaning that $\rho\approx 1$, the occasions
when the system is either empty or full are rare (i.e.\ both $\pi_0$ and
$\pi_J$ are `small') so that mostly periods with above or below average
net arrivals are balanced by periods of above or below average productive
service, and `consequently' the variability of throughput of the system
as measured by the output $N_\dep$ is to some extent `smoothed'.
This intuitive explanation of the BRAVO effect is one we have not been able
to translate into a mathematical explanation.

In other work concurrently under preparation, we have examined the
many-server  $M/M/s$ system with reneging to see whether it too exhibits the
BRAVO effect.
%If the system also has a buffer, much as in this paper, then the output of the system exhibits the BRAVO effect, but with an infinite buffer the effect vanishes.  We found this surprising because we know that a BRAVO effect exists for queueing systems without buffers (Hanbali {\sl et al.}, 2011); why should the stabilizing effect of a reneging mechanism, where the process $Q_s(\cdot)$ is affected only when $Q_s(\cdot)>s$, quench any `internal compensations' which, presumably, may underlie it?
%
In studying a system with reneging, the departure process $N_\dep(0,\cdot\,]$ no longer
consists precisely of the deaths in the birth--death process $Q_s(\cdot)$ but only
of a subset of those deaths.  Consequently the second equality at (1.5) is
no longer available as an expression for $\RRR_\pi$, but an independent
computation, with the advantage from an expository viewpoint that it relies
on more primitive results, yields a substitute expression
\begin{equation}
\label{eq:Rfuture}
1 - 2 \sum_{i=0}^J (P_i - \Lambda^*_i)\Big(q_{i+1} - \frac{\lambda^*}{\pi_i \lambda_i}  (P_i - \Lambda^*_i) \Big),
\end{equation}
where $q_{i}$ is the probability of having a death at state $i$ of $Q_s(\cdot)$ count as an increment of $N_\dep(0,\cdot]$. A complete derivation of
this result has been deferred from this paper, in which context the present work should have been more self-contained.

The same substitute expression
can also cope with a sequence of systems with balking, where we expect the
BRAVO effect may still be observed, but whether such a system must be finite,
with a buffer
in place as well as the balking mechanism, remains to be discovered.

%%%%%%%%%%%%%%%%%%%%%
%% ACKNOWLEDGMENTS %%
%%%%%%%%%%%%%%%%%%%%%
\vspace{10pt}
{\bf Acknowledgment}: DJD's work done while an Honorary Professorial Associate at the University
of Melbourne. JvL is supported by a Starting Grant of the European Research Council. YN's work done in part while at the Department of Mathematics at the Swinburne University of Technology. YN is supported by Australian Research Council (ARC) grants DP130100156 and DE130100291.

%%%%%%%%%%%%%%%%%%%%%%%%%%%%%%%%%%%%%%%%%%%%%%%%%%
%%%%%%%%%%%%%%%%%%%%%%%%%%%%%%%%%%%%%%%%%%%%%%%%%%
%%%%%%%%%%%%%%%%%%%%%%%%%%%%%%%%%%%%%%%%%%%%%%%%%%
%%%%%%%%%%%%%%%%%%%%%%%%%%%%%%%%%%%%%%%%%%%%%%%%%%
%%%%%%%%%%%%%%%%%%%%%%%%%%%%%%%%%%%%%%%%%%%%%%%%%%
%%%%%%%%%%%%%%%%%%%%%%%%%%%%%%%%%%%%%%%%%%%%%%%%%%
%%%%%%%%%%%%%%%%%%%%%%%%%%%%%%%%%%%%%%%%%%%%%%%%%%
%%%%%%%%%%%%%%%%%%%%%%%%%%%%%%%%%%%%%%%%%%%%%%%%%%
%%%%%%%%%%%%%%%%%%%%%%%%%%%%%%%%%%%%%%%%%%%%%%%%%%
\appendix

%%%%%%%%%%%%%%%%%%%%%%%%%%%%%%%%%%%%%%%%%%%%%%%%%%
%%%%%%%%%%%%%%%%%%%%%%%%%%%%%%%%%%%%%%%%%%%%%%%%%%
%%%%%%%%%%%%%%%%%%%%%%%%%%%%%%%%%%%%%%%%%%%%%%%%%%
%%%%%%%%%%%%%%%%%%%%%%%%%%%%%%%%%%%%%%%%%%%%%%%%%%
%%%%%%%%%%%%%%%%%%%%%%%%%%%%%%%%%%%%%%%%%%%%%%%%%%
%%%%%%%%%%%%%%%%%%%%%%%%%%%%%%%%%%%%%%%%%%%%%%%%%%
%%%%%%%%%%%%%%%%%%%%%%%%%%%%%%%%%%%%%%%%%%%%%%%%%%
%%%%%%%%%%%%%%%%%%%%%%%%%%%%%%%%%%%%%%%%%%%%%%%%%%
%%%%%%%%%%%%%%%%%%%%%%%%%%%%%%%%%%%%%%%%%%%%%%%%%%
\section{Asymptotic normality of Poisson probabilities}
\label{sec:appPoissonNormal}

In formulating the condition at \eqref{scaling2} and establishing the double limit property \eqref{double}  we use both local and central limit properties of the Poisson distribution with mean $\kappa$,
here denoted $\{\varpi_i\} := \{\varpi_i(\kappa)\} :=  \{\e^{-\kappa}\kappa^i/i!\}$ and write $\Pi_i(k) := \sum_{j=0}^j \varpi_j(\kappa)$ for the corresponding distribution function.
For convenience we state these properties with their respective rates of convergence in two lemmas below.  We draw on the exposition in Feller \cite{feller1968ipt} concerning Stirling's formula and limit results for the binomial distribution (p.54 and p.183 respectively).

For Stirling's formula Feller shows that for any positive integer, $s$,
\eqan{
\e^{1/(12s+1)} < \frac{s!}{\sqrt{2\pi s}\,(s/\e)^s} < \e^{1/(12s)}\,,
 }
from which it follows that
\eqan{
\varpi_s(s) = \e^{-s} \frac{s^s}{s!} = \frac{1}{\sqrt{2\pi s}\,}\Big(1 -
\frac{\gamma_s}{12s}\Big),
 }
where $0 < \gamma_s \to 1$ as $s\to\infty$; indeed, Feller notes that the ratio
\eqan{
\frac{s!}{\sqrt{2\pi s}\, (s/\e)^s \e^{1/12s}}\break
  = 1 - (1/12s^2) \big(1+o(1)\big).
 }
We use a local central limit theorem for Poisson probabilities, including
a rate of convergence, as follows.
\begin{lemma}
\label{lemma:pois1}
For $i,s\to\infty$ such that $i/\sqrt{s} \to \gamma$ for
fixed finite $\gamma>0$, uniformly for $i<\lfloor s^{5/8}\rfloor$,
\eqan{
\frac{\varpi_{s-i}(s)}{\varpi_s(s)} 
   = \frac{\phi(\gamma)}{\phi(0)}\,\big(1 + O(1/s^{1/8})\big),
 }
where the rate of convergence holds uniformly for $i\le s^{5/8}$.
\end{lemma}
\begin{proof}
Analogous to Feller's proof of the local asymptotic normality of binomial probabilities, write
\eqan{
\label{eq:omegaRatio2}
\frac{\varpi_s(s)}{\varpi_{s-i}(s)}
   = \frac{s^i}{s(s-1)\cdots (s-i+1)}
  \,=\, \frac{1}{(1-t_1)\cdots(1-t_{i-1})}\,,
 }
where for $j=1,\ldots,s-1$, $t_j = j/s$.
For $|t| \le \half$, $1/(1-t)=\exp[-\log(1-t)] = \e^{t + \half t^2 + \vartheta_i t^3}$,
%where $|\vartheta_t|\le \frac{1}{3}$.
%$\e^{t + \vartheta_t t^2}$,
where $|\vartheta_t|<1$.
Then for positive integers $i\le\half s$, the ratio at \eqref{eq:omegaRatio2} equals
\eqan{
\label{eq:expQuantLemma}
%\exp\Big(\frac{(i-1)i}{2s}
% + \frac{(i-1)i(2i-1)}{12 s^2}
%+ \overline{\vartheta_j} \,\frac{i^2(i-1)^2)}{4s^3} \Big),
  \exp\Big( \frac{(i-1)i}{2s} + \overline{\vartheta_j}\,
   \frac{(i-1)i(2i-1)}{6s^2} \Big),
 }
where $0<\overline{\vartheta_i} \le 1$.
Let $j_s\to\infty$ as $s\to\infty$ such that $j_s^3/s^2 \to 0$.  Then
 for $(i,s)$ with $i< j_s$, the quantity at \eqref{eq:expQuantLemma} differs from
$\e^{\half \gamma^2}$ by a multiplicative factor that is dominated by $\exp\big(\overline{\vartheta_i} (j_s^3/s^2)\big)$, so for $j_s = s^{5/8}$, this last term yields a factor $O(1/s^{1/8})$, whose expansion then completes the proof.
\end{proof}

\begin{lemma}
 \label{lemma:pois2}
 There exist finite constants $\psi_{sj}$ that are uniformly bounded in both $j$ and $s$ such that
\eqan{
\label{eq:sumOmega}
\sum_{i=0}^j \varpi_i(s) = \Phi\Big(\frac{j-s}{\sqrt{s}\,}\Big)
   + \frac{\psi_{sj}}{\sqrt{s}}\,.
}
\end{lemma}

\begin{proof}
Let $X_s$ denote a Poisson r.v. with mean $s$, so that $X_s$ is expressible as the sum of $s$ i.i.d. Poisson r.v.s $X_1$ with mean 1, variance 1, and third absolute moment about the mean $\beta_3 = 1 + 2/\e$.  Then by the Berry--Esseen theorem (e.g.\ \cite{shiryayev1984probability} p.342), for a positive constant $C<4/5$,
\eqan{
\sup_i \Big|\Pr\{X_s \le i\} - \Phi\Big(\frac{i-s}{\sqrt{s}}\Big)\Big|
   \le \frac{C\beta_3}{\sqrt{s}}\,.
 }
\eqref{eq:sumOmega} follows, with an explicit bound for the constant, and we write this
in the form, for uniformly bounded quantities $\psi_{si}$,
\eqan{
\label{eq:piRel3}
\Pi_i(s) = \Phi\Big(\frac{i-s}{\sqrt{s}}\Big) + \frac{\psi_{si}}{\sqrt{s}}\,.
 }
\end{proof}

\begin{corollary}
The constants $\{\psi_{si}\}$ in the Berry--Esseen expansion \eqref{eq:piRel3} for the Poisson distribution satisfy $\sum_{i=1}^s \psi_{si} = o(s)$.
\end{corollary}
\begin{proof}
From around (3.10),
\eqan{
S_{s,s-1} = \sum_{i=0}^{s-1} \varpi_s(s) \Pi_i(s) \,=\, s\varpi_s^2
   \,=\, \frac{1}{2\pi}\Big(1 + \frac{\vartheta_s}{12s}\Big)^2
   \,\to\,\frac{1}{2\pi}\qquad(s\to\infty),
 }
where $|\vartheta_s|\le 1$,
while substituting from \eqref{eq:piRel3} into the expression for $S_{s,s-1}$ gives
\begin{align}
S_{s,s-1} &= \sum_{i=0}^{s-1} \varpi_s(s)
    \bigg[\Phi\Big(\frac{i-s}{\sqrt{s}\,}\Big)
    + \frac{\psi_{si}}{\sqrt{s}\,}\bigg] \\
\label{eq:sum2Lemma}
   &= (1 + \vartheta_s/12s)
  \bigg[ \sum_{i=1}^s \Phi(-x_{si}) \frac{1}{\sqrt{2\pi}\,\sqrt{s}}
   + \frac{1}{\sqrt{2\pi}\,s} \sum_{i=0}^{s-1} \psi_{si} \bigg]\,, &
\end{align}
where $x_{si}=i/\sqrt{s}$. Examining these last two sums, the former is an approximation, based on the $s$ intervals determined by the
$s+1$ points $\{i/\sqrt{s}\}$ $(i=0,\ldots,s)$, to the integral
$\int_0^{\sqrt{s}} \Phi(-u) \d u/\sqrt{2\pi}$.  This in turn is a finite
 portion of the integral $\int_0^\infty \Phi(-u)\,\d u/\sqrt{2\pi}$; this  improper (Riemann) integral exists because $\Phi(u)$ is directly Riemann integrable on $(-\infty,0)$ and equals $1/\sqrt{2\pi}\,$.  This means that the Riemann integral approximations converge to $1/2\pi = \lim_{s\to\infty} S_{s,s-1}$. This last equality implies that the last sum at \eqref{eq:sum2Lemma} $\to0$ for $s\to\infty$.
\end{proof}

The Berry--Esseen theorem is typically proved using Fourier techniques. We have not explored how such techniques would, presumably, provide another way of establishing the corollary. Analogues of Lemmas \ref{lemma:pois1} and \ref{lemma:pois2} for general Poisson distributions $\{\varpi_i(\lambda)\}$ for (large) $\lambda$, not necessarily an integer, are readily constructed.

%%%%%%%%%%%%%%%%%%%%%%%%%%%%%%%%%%%%%%%%%%%%%%%%%%
%%%%%%%%%%%%%%%%%%%%%%%%%%%%%%%%%%%%%%%%%%%%%%%%%%
%%%%%%%%%%%%%%%%%%%%%%%%%%%%%%%%%%%%%%%%%%%%%%%%%%
%%%%%%%%%%%%%%%%%%%%%%%%%%%%%%%%%%%%%%%%%%%%%%%%%%
%%%%%%%%%%%%%%%%%%%%%%%%%%%%%%%%%%%%%%%%%%%%%%%%%%
%%%%%%%%%%%%%%%%%%%%%%%%%%%%%%%%%%%%%%%%%%%%%%%%%%
%%%%%%%%%%%%%%%%%%%%%%%%%%%%%%%%%%%%%%%%%%%%%%%%%%
%%%%%%%%%%%%%%%%%%%%%%%%%%%%%%%%%%%%%%%%%%%%%%%%%%
%%%%%%%%%%%%%%%%%%%%%%%%%%%%%%%%%%%%%%%%%%%%%%%%%%

\bibliography{PaperDatabase,BookDatabase}
\end{document}